\documentclass{article}
\usepackage{Khanh}
\usepackage[backend=biber,style=alphabetic,sorting=nty ]{biblatex}

\newtheorem{thm}{Theorem}[section]
\newtheorem{defi}[thm]{Definition}

\newtheorem{lem}[thm]{Lemma}

\newtheorem{conj}[thm]{Conjecture}
\newtheorem{prop}[thm]{Proposition}

\newtheorem{quest}[thm]{Question}

\newtheorem{remark}[thm]{Remark}

\makeatletter
\renewenvironment{proof}[1][\proofname]{%
   \par\pushQED{\qed}\normalfont%
   \topsep6\p@\@plus6\p@\relax
   \trivlist\item[\hskip\labelsep\bfseries#1\@addpunct{.}]%
   \ignorespaces
}{%
   \popQED\endtrivlist\@endpefalse
}

\DeclareMathOperator{\trans}{trans}
\DeclareMathOperator{\aug}{aug}
\DeclareMathOperator{\EV}{EV}

\title{Orderability}
\author{Khanh Le}
\date{\today}

\title{Left orderability for surgeries on the $[1,1,2,2,2j]$ two-bridge knots}
\author{Khanh Le}
\date{March 2021}

\begin{document}

\maketitle

\begin{abstract}
    Let $M$ be a $\bbQ$-homology solid torus. In this paper, we give a cohomological criterion for the existence of an interval of left-orderable Dehn surgeries on $M$. We apply this criterion to prove that the two-bridge knot that corresponds to the continued fraction $[1,1,2,2,2j]$ for $j\geq 1$ admits an interval of left-orderable Dehn surgeries. This family of two-bridge knots gives some positive evidence for a question of Xinghua Gao. 
\end{abstract}

\section{Introduction}
A group $G$ is \emph{left orderable} if it admits a strict total ordering on the group elements such that $g < h $ implies that $fg < fg$ for all elements $f,g,h \in G$. Left-orderability arises naturally in the study of low-dimensional topology, foliation theory, and group theory. Well-known examples of left-orderable groups include torsion-free abelian groups, free (non-abelian) groups, surface groups and the group of orientation preserving homeomorphisms of the real line. In 3-manifold topology, left-orderability is an important concept due to its role in the L-space conjecture.
\begin{conj}[The L-space conjecture]
\label{conj:Lspace}
For an irreducible $\bbQ$-homology 3-sphere $M$, the following are equivalent 
\begin{enumerate}
    \item $\pi_1(M)$ is left-orderable.
    \item $M$ is not an L-space.
    \item $M$ admits a coorientable taut foliation. 
\end{enumerate}
\end{conj}
An L-space is a $\bbQ$-homology 3-sphere with $\dim \hat{HF}(M) = |H_1(M;\bbZ)|$ where $\hat{HF}(M)$ is the Heegaard Floer homology of $M$ \cite[Definition 1.1]{OS05}. There has been a substantial amount of evidence in favor of this conjecture. For example, the L-space conjecture holds for all graph manifolds \cite{BC17} and \cite{RRDW15}. 

In view of \cref{conj:Lspace}, there have been a lot of ideas developed to study left-orderability of 3-manifold groups. It is a well-known fact that a countable group is left-orderable if and only if it embeds in the group of orientation-preserving homeomorphisms of the real line \cite[Theorem 6.8]{Ghys}. In the case of an irreducible compact 3-manifold, its fundamental group is left-orderable if and only if it admits a non-trivial homomorphism onto a left-orderable group \cite[Theorem 1.1]{BRW05}. In particular, all manifolds with positive first Betti number are left-orderable. Therefore, it is interesting to construct left orders on $\bbQ$-homology spheres, for example those coming from doing Dehn filling and from taking cyclic branched covering of $\bbQ$-homology solid torus. 

A fruitful way to build left-orderings on $\bbQ$-homology spheres is by lifting $\PSL_2(\bbR)$ representations to $\widetilde{\PSL}_2(\bbR)$. This strategy has been employed with a lot of success, for example see \cite{TranBranchedCover, TranDehnSurgery, HuBranchedCover}. Recently, Dunfield, Culler and independently Gao have introduced the idea of using the extension locus of a compact 3-manifold with torus boundary $M$ to order families of $\bbQ$-homology spheres arising by doing Dehn filling on $M$ \cite{CD18,gao2019orderability}. Furthermore, they gave several criteria implying the existence of intervals of left-orderable Dehn fillings on $M$. To state their results, we need the following definition:

\begin{defi}
\label{def:LongitudinalRigidity}
A compact 3-manifold $Y$ has \emph{few characters} if each positive-dimensional component of the $\PSL_2(\bbC)$-character variety $X(Y)$ consists entirely of characters of reducible representations. An irreducible $\bbQ$-homology solid torus $M$ is called \emph{longitudinally rigid} when $M(0)$ has few characters where $M(0)$ is the closed manifold obtained from $M$ by doing Dehn filling along the homological longitude.  
\end{defi}
We summarize their results in the following:

\begin{thm} 
[{{\cite[Theorem 7.1 ]{CD18} and \cite[Theorem 5.1]{gao2019orderability}}}]
\label{thm:CullerDunfieldGao}
Suppose that $M$ is longitudinally rigid irreducible $\bbZ$-homology solid torus. Then the following are true:
\begin{enumerate}
    \item If the Alexander polynomial of $M$ has a simple root $\xi \neq 1$ on the unit circle, then there exists $a >0$ such that for every rational $r \in (-a,0) \cup (0,a)$ the Dehn filling $M(r)$ is orderable. 
    \item If the Alexander polynomial of $M$ has a simple positive real root $\xi \neq  1$, then there exists a nonempty interval $(-a,0]$ or $[0,a)$ such that for every rational $r$ in the interval, the Dehn filling $M(r)$ is orderable. 
\end{enumerate}
\end{thm}

\begin{remark}
In fact, their techniques also apply to the case where $M$ is a $\bbQ$-homology solid torus with some further hypothesis on $\xi$ in the first statement. The full version of the second statement is stated below in \cref{thm:Gao-Ordering}. 
\end{remark}

Culler and Dunfield also proved the following criterion for left-orderability:
\begin{thm}
\label{thm:CullerDunfieldTraceFieldOrdering}
Suppose that $M$ is a hyperbolic $\bbZ$-homology solid torus, whose trace field has a real embedding, then there exists $a >0$ such that for every rational $r\in(-a,0) \cup (0,a)$ the Dehn filling $M(r)$ is orderable. 
\end{thm}

In view of \cref{thm:CullerDunfieldGao}, it is natural to ask when a $\bbQ$-homology solid torus is longitudinally rigid. Since the character variety is notoriously hard to compute, see for example \cite{BP13, Chesebro20}, longitudinal rigidity is difficult to study in a general setting. Culler and Dunfield gave a topological condition which implies longitudinal rigidity. In particular, they introduced the following concept:

\begin{defi}
\label{def:Lean}
Let $M$ be a knot exterior. We say that $M$ is \emph{lean} if the longitudinal Dehn filling $M(0)$ is prime and every closed essential surface in $M(0)$ is a fiber in a fibration over $S^1$.
\end{defi}

For example, the $(-2,3,2s+1)$-pretzel knots were shown to be lean for $s\geq 3$, so there is an interval of about $0$ of left-orderable Dehn surgeries on these knot complements \cite[Theorem 4]{Nie19}. However as remarked in \cite[Section 1.6]{CD18}, this leanness condition is rather restrictive. In particular, for a knot complement $K$ in $S^3$ being lean implies that $K$ fibers. Nevertheless, the first statement of \cref{thm:CullerDunfieldGao} was proved to be true without the condition of longitudinal rigidity by Herald and Zhang \cite[Theorem 1]{HZ19}. Motivated by this result, Xinghua Gao asked:

\begin{quest}\cite[Section 7]{gao2019orderability}
\label{quest:GaoQuestion}
Can the longitudinal rigidity condition be dropped from the second statement of \cref{thm:CullerDunfieldGao}? Is it possible to prove $H^1(\pi_1(M(0));\liesl_2(\bbR)_{\rho}) = 0$ where $\rho$ is the non-abelian reducible representation of $\pi_1(M)$ coming from the root of the Alexander polynomial?
\end{quest}

Following the suggestion in this question, we say that $\bbQ$-homology solid torus is locally longitudinal rigid at a root of the Alexander polynomial if $H^1(\pi_1(M(0));\liesl_2(\bbC)_{\rho}) = 0$ where $\rho$ is a non-abelian reducible representation coming from this root, see \cref{def:loclongitudinallyrigid} for a precise definition. In fact, we prove that the second item of \cref{thm:CullerDunfieldGao} still holds true under this weakened hypothesis.      
\begin{thm}
\label{thm:Le-Ordering}
Suppose that $M$ is an irreducible $\bbQ$-homology solid torus and that the Alexander polynomial of $M$ has a simple positive real root $\xi \neq 1$. Furthermore, suppose that $M$ locally longitudinally rigid at $\xi$. Then there exists a nonempty interval $(-a,0]$ or $[0,a)$ such that for every rational $r$ in the interval, the Dehn filling $M(r)$ is orderable.
\end{thm}

As an application, we apply this result to produce an interval of left-orderable Dehn surgeries on an infinite family of two-bridge knots complement.   

\begin{thm}
\label{thm:LOTwoBridge}
For every two-bridge knot $K_j$ corresponding to the contiued fraction $[1,1,2,2,2j]$ where $j\geq 1$, there exists a nonempty interval $(-a,0]$ or $[0,a)$ such that for every rational $r$ in the interval, the Dehn filling $M(r)$ is left-orderable.  
\end{thm}

\begin{remark}
As we will see in \cref{lem:AlexanderPoly}, the Alexander polynomial of $K_j$ has all simple positive real roots that are not 1, and is not monic for $j \geq 2$. In particular, the complement of $K_j$ is not lean for $j\geq 2$. Furthermore, the trace field of $K_j$ for $1 \leq j \leq 30 $ has no real places, and it is most likely that the trace fields of all knots in this family share this property. Therefore, \cref{thm:LOTwoBridge} is not a direct consequence of \cref{thm:CullerDunfieldGao} nor \cref{thm:CullerDunfieldTraceFieldOrdering}. The family of two-bridge knots $[1,1,2,2,2j]$ is a genuinely new family of knots with an interval left-orderable Dehn surgeries which cannot be obtained from prior techniques. 
\end{remark}

\subsection{Outline}
In \cref{sec:Preliminaries}, we review some background materials on group cohomology, $\PSL_2$-representation variety, formal deformation of representation and holonomy extension locus. At the end of this section, we will give a proof of \cref{thm:Le-Ordering}. In \cref{sec:TwoBridge}, we will carry out the group cohomology calculation and prove that the complement of $K_j$ is locally longitudinally rigid at all roots of the Alexander polynomial. As a result, \cref{thm:LOTwoBridge} will follow from \cref{thm:Le-Ordering}.

\section{Preliminaries}
\label{sec:Preliminaries}
\subsection{Group cohomology and $\PSL_2$-representation variety}

Following the notation in \cite{CD18} and \cite{gao2019orderability}, we set $G = \PSL_2(\bbR)$ and $G_\bbC = \PSL_2(\bbC)$ throughout the paper. For a compact manifold $M$ and a group $H$, we let $R_H(M) = \Hom(\pi_1(M),H)$ be the representation variety. When $H = G_\bbC$, we denote $R(M) := R_{G_\bbC}(M)$. Since $\pi_1(M)$ is finitely generated, $R(M)$ can be identified with an algebraic subset in some affine space $\bbC^N$. The group $G$ acts on $R(M)$ by conjugation. Let us consider the minimal Hausdorff quotient $X(M) := R(M)//G_\bbC$ and the quotient map $\pi:R(M) \to X(M)$. Given a representation $\rho \in R(M)$, a character of $\rho$ is the map $\chi_\rho:\pi(M) \to \bbC$ defined by $\chi_\rho(\gamma) = \tr^2(\rho(\gamma))$. By \cite[Theorem 1.3]{HP04}, there exists a bijection between between the points of $X(M)$ and the characters of representations in $R(M)$ such that the point $t(\rho) = [\rho]$ corresponds to $\chi_\rho$. Therefore, we refer to $X(M)$ as the $\PSL_2(\bbC)$-character variety of $M$. 

Let $\Gamma$ be a group and $\rho:\Gamma \to G_\bbC$ be a representation. The Lie algebra of $G_\bbC$ can be identified with the set of trace-less 2-by-2 matrices over $\bbC$. Using the adjoint representation, the Lie algebra $\liesl_2(\bbC)$ becomes a $\Gamma$-module by
\begin{equation*}
    \gamma\cdot v = \rho(\gamma) v \rho(\gamma)^{-1}.
\end{equation*}
We denote this $\Gamma$-module by $\liesl_2(\bbC)_\rho$. The space of 1-cocycles is 
\begin{equation*}
\label{eq:SpaceCocycles}
    Z^1(\Gamma;\liesl_2(\bbC)_\rho) = \{ z: \Gamma \to \liesl_2(\bbC) \mid z(\gamma\gamma') = z(\gamma)  + \gamma \cdot z(\gamma') \ \forall \gamma, \gamma' \in \Gamma\}.
\end{equation*}
Alternatively when $\Gamma$ is a finitely presented group, we can also describe the space of cocycles as maps $\Gamma \to \liesl_2(\bbC)$ satisfying the group relations of $\Gamma$. In particular, suppose that $\Gamma = \langle \gamma_1,\dots,\gamma_n \mid w_1(\gamma_i),\dots,w_k(\gamma_i)\rangle$ is a finite presentation and that $z(\gamma_i) = v_i \in \liesl_2(\bbC)$. Given any element $w \in\Gamma$, we can express $w$ as a word $w(\gamma_i)$ in the generators $\gamma_i$'s of $\Gamma$. The equation      

\begin{equation}
\label{eq:FoxDerivatives}
    z(\gamma\gamma') = z(\gamma)  + \gamma \cdot z(\gamma') \quad \forall \gamma, \gamma' \in \Gamma
\end{equation}
determines the image of $z(w)$. This gives us a well-defined cocycle on $\Gamma$ if and only if $z(w_j) = 0$ for all relations $w_j$ of $\Gamma$, see \cite[Equation 4]{W64}. The space of 1-coboundaries is 
\begin{equation}
\label{eq:Coboundary}
    B^1(\Gamma; \liesl_2(\bbC)_\rho) = \{b: \Gamma \to \liesl_2(\bbC) \mid \exists \ v \in \liesl_2(\bbC), \ b(\gamma) = (\gamma - 1_\Gamma)\cdot v\}.
\end{equation}
Finally, the group cohomology is defined by 
\begin{equation*}
\label{eq:Cohomology}
H^1(\Gamma;\liesl(\bbC)_\rho) =  Z^1(\Gamma;\liesl_2(\bbC)_\rho)/    B^1(\Gamma; \liesl_2(\bbC)_\rho).   
\end{equation*}

\begin{defi}
\label{def:ZariskiTangentSpace}
Suppose that $V$ is an affine algebraic variety in $\bbC^n$. Let \[I(V) = \{ f \in \bbC[x_1,\dots,x_n] \mid f(x) = 0 \quad \forall x \in V\}\] be the vanishing ideal of $V$. Define the Zariski tangent space to $V$ at $p$ to be the vector space of derivatives of polynomials.
\begin{equation*}
    T^{Zar}_p(V) = \left\{ \frac{d\gamma}{dt}\rvert_{t=0}\in \bbC^n \mid \gamma \in (\bbC[t])^n, \gamma(0) = p \text{ and } f\circ \gamma \in t^2 \bbC[t] \ \forall f \in I(V)\right\}.
\end{equation*}
\end{defi}

It was observed by Weil in \cite{W64} that for any Lie group $H$ and $\rho \in R_H(M)$ the Zariski tangent space embeds in the space of 1-cocycles $Z^1(\pi_1(M);\mathfrak{h})$ where $\mathfrak{h}$ is the Lie algerba of $H$. In particular, we have the following inequalities

\[
\dim Z^1(\Gamma;\liesl_2(\bbC)_\rho) \geq \dim  T^{Zar}_\rho(R(\Gamma)).
\]

\subsection{Formal deformation of representation}

We will review some background materials on formal deformations of representations and integrability of cocycles. The concept of integrable cocycles will be important to building a certain path of representations required in the proof of \cref{thm:Le-Ordering}, see also \cref{lem:AllCocyclesAreIntegrable}. For this discussion, let $\Gamma$ be a finitely presented group, $A_k := \bbR[[t]]/(t^{k+1})$ for $k\in \mathbb{N}$ and $A_\infty := \bbR[[t]]$. Consider the following groups $G_k := \PSL_2(A_k)$ and $G_\infty := \PSL_2(A_\infty)$.     

\begin{defi}
Let $\rho:\Gamma \to G$ be a representation. A \emph{formal deformation of $\rho$} is a representation $\rho_\infty:\Gamma \to G_\infty$ such that $\rho = p_0 \circ \rho_\infty$ where $p_0: G_\infty \to G$ is the homomorphism induced by evaluating the formal power series at $t=0$. 
\end{defi}

For any formal deformation $\rho_\infty:\Gamma \to G_\infty$ of $\rho$, we can write 
\begin{equation}
\label{eq:FormalDef}
\rho_\infty(\gamma) = \exp\left(\sum_{i=1}^\infty t^i u_i(\gamma)\right) \rho(\gamma)
\end{equation}
where $u_i:\Gamma \to \liesl_2(\bbR)_\rho \in C^1(\Gamma;\liesl_2(\bbR)_\rho)$. Since $\rho_\infty$ is a homomorphism, a calculation using the Taylor series for the exponential map implies that $u_1 \in Z^1(\Gamma;\liesl_2(\bbR)_\rho)$. Conversely, we have the following definition:

\begin{defi}
A cocycle $u_1\in Z^1(\Gamma;\liesl_2(\bbR)_\rho)$ is \emph{integrable} if there exists a formal deformation $\rho_\infty$ of $\rho$ given by \cref{eq:FormalDef}. In this case, we say that $\rho_\infty$ is a formal deformation of $\rho$ with leading term $u_1$.  
\end{defi}

Given a representation $\rho: \Gamma \to G$ and a cocycle  $u_1 \in Z^1(\Gamma;\liesl_2(\bbR)_\rho)$, the existence of a formal deformation of $\rho$ with leading term $u_1$ is equivalent to the vanishing of a series of obstruction classes in $H^2(\Gamma;\liesl_2(\bbR)_\rho)$ \cite[Proposition 3.1 and Corollary 3.2]{HPS}. In particular, we have the following proposition from \cite{HPS}:

\begin{prop}
\label{prop:FormalDeformationAndObstructionClass}
Let $\rho \in R_G(\Gamma)$ and $u_i \in C^1(\Gamma;\liesl_2(\bbR)_\rho)$ for $1 \leq i \leq k$ be given. Suppose that we have constructed a representation $\rho_k:=\rho_k^{(\rho;u_1,\dots,u_k)}:\Gamma \to G_k$ given by
\[
\rho_k(\gamma) = \exp\left(\sum_{i=1}^kt^i u_i(\gamma)\right)\rho(\gamma) \mod t^{k+1}.
\]
There exists an obstruction class $\zeta_{k+1}:=\zeta_{k+1}^{(u_1,\dots,u_k)} \in H^2(\Gamma;\liesl_2(\bbR)_\rho)$ with the following properties:
\begin{enumerate}
    \item There is a cochain $u_{k+1}$ such that $\rho_{k+1}^{(\rho;u_1,\dots,u_{k+1})}: \Gamma \to G_{k+1}$ given by \[\rho_{k+1}^{(\rho;u_1,\dots,u_{k+1})} = \exp\left(\sum_{i=1}^{k+1}t^i u_i(\gamma)\right)\rho(\gamma) \mod t^{k+2}\] is a homomorphism if and only if $\zeta_{k+1} = 0$.
    \item The obstruction $\zeta_{k+1}$ is natural in the following sense: if $f:\Gamma' \to \Gamma$ is a homomorphism then
    \[
    f^*\rho_{k}^{(\rho;u_1,\dots,u_k)} = \rho_{k}^{(f^*\rho;f^*u_1,\dots,f^*u_k)}
    \]
    is a homomorphism and $f^*\zeta_{k+1}^{(u_1,\dots,u_k)} = \zeta_{k+1}^{(f^*u_1,\dots,f^*u_k)}$. 
\end{enumerate}
Consequently, an infinite sequence $\{u_i\}_{i=1}^\infty \subset C^1(\Gamma;\liesl_2(\bbR))_\rho$ defines a formal deformation of $\rho$, $\rho_\infty: \Gamma\to G_\infty$ via \cref{eq:FormalDef} if and only if $u_1$ is a cocycle and $\zeta_{k+1}^{(u_1,\dots,u_k)} = 0$ for all $k \geq 1$.  
\end{prop}

\begin{remark}
\cref{prop:FormalDeformationAndObstructionClass} was stated over $\bbC$ in \cite{HPS}. Since the construction of the obstruction $\zeta_{k+1}$, see \cite[Definition 3.4]{HPS}, and the proof of \cref{prop:FormalDeformationAndObstructionClass} is purely homological, it remains true over $\bbR$. 
\end{remark}


 


\subsection{Holonomy extension locus}
Now we recall some definitions and results about the holonomy extension locus from \cite{gao2019orderability}. The group $G$ acts on $P^1_\bbC$ by Mobius transformation leaving $P^1_\bbR$ invariant. Any nontrivial abelian subgroup of $G$ either contains only parabolic elements and has one fixed point in $P^1_\bbC$ or contains only hyperbolic or elliptic elements and has two fixed points in $P^1_\bbC$. Let $\widetilde{G} = \widetilde{\PSL}_2(\bbR)$ be the universal covering group of $G$. The group $\widetilde{G}$ also acts on $P^1_\bbC$ by pulling back the action of $G$. We say that an element $\widetilde{g}\in\widetilde{G}$ is hyperbolic, parabolic, elliptic, or central if the image of $\widetilde{g}$ in $G$ is hyperbolic, parabolic, elliptic, or trivial, respectively.

We denote by $M$ a compact 3-manifold with a single torus boundary component and define the augmented representation $R_G^{\aug}(M)$ as follows. Since abelian subgroups of $G$ act with global fixed points, we define the augmented representation variety $R_G^{\aug}(M)$ to be the subvariety of $R_G(M) \times P^1_\bbC$ consisting of pairs $(\rho,z)$ where $z$ is a fixed point of $\rho(\pi_1(\partial M))$. Since the action of $\widetilde{G}$ on $P^1_{\bbC}$ comes from pulling back the action of $G$, we can also define $R^{\aug}_{\widetilde{G}}(M)$ to be the real analytic subvariety of $R_{\widetilde{G}}(M)\times P^1_\bbC$ consisting of pairs $(\rho,z)$ where $z$ is a fixed point of $\rho(\pi_1(\partial M))$. Similarly, we define $R^{\aug}_{\widetilde{G}}(\partial M)$ to be the real analytic subvariety of $R_{\widetilde{G}}(\partial M)\times P^1_\bbC$ consisting of pairs $(\rho,z)$ where $z$ is a fixed point of $\rho(\pi_1(\partial M))$.

Given a hyperbolic, parabolic or central element $\tilde{g} \in \widetilde{G}$ with a fixed point $v \in P^1_\bbC$, let $g\in G$ be the image of $\widetilde{g}$ and $a$ be a square root of the derivative of $g$ at $v$. We define  \[\ev(\widetilde{g},v) := (\ln(|a|),\trans(\widetilde{g}))\] where $\trans:\widetilde{G} \to \bbR$ is the translation number given by 
\[
\trans(\widetilde{g})=\lim_{n\to \infty} \frac{\widetilde{g}^n(0)}{n}.
\]
for some $x \in\bbR$. This limit exists for all $\widetilde{g} \in \widetilde{G}$, see \cite[Section 5.1]{Ghys}. It is shown in \cite[Lemma 3.1]{gao2019orderability} that $\ev(-,v)$ is a homomorphism when restricted to hyperbolic or parabolic abelian subgroups of $\widetilde{G}$ fixing $v$. We get a group homomorphism

\[
\ev(\widetilde{\rho}(-),v) : \pi_1(\partial M) \to \bbR \times \bbZ 
\]
for $\widetilde{\rho} \in R^{\aug}_{\widetilde{G}}(\partial M)$ whose image in $\widetilde{G}$ is hyperbolic, parabolic or central. In other words, we can view $\ev(\widetilde{\rho}(-),v)$ as an element of $\Hom(\pi_1(\partial M),\bbR \times \bbZ)$. We are now ready to define the holonomy extension locus.

\begin{defi}
\label{def:EV}
Let $PH_{\widetilde{G}}(M)$ be the subset of $R_{\widetilde{G}}^{\aug}(M)$ whose restriction to $\pi_1(\partial M)$ is either hyperbolic, parabolic or central. Consider the restriction map $i^*: R^{\aug}_{\widetilde{G}}(M)  \to R^{\aug}_{\widetilde{G}}(\partial M)$ induced by the inclusion $i : \partial M \to M$. Define $\EV: i^*(PH_{\widetilde{G}}(M)) \to H^1(\partial M;\bbR) \times H^1(\partial M;\bbZ) $ by 
\begin{equation*}
    (\widetilde{\rho},v) \mapsto \ev((\widetilde{\rho}(-),v)).
\end{equation*}
\end{defi}

\begin{defi}
\label{def:HolonomyExtensionLocus}
Consider the composition 
\begin{equation*}
    PH_{\widetilde{G}}(M) \subset R_{\widetilde{G}}^{\aug}(M) \xrightarrow{i^*} R_{\widetilde{G}}^{\aug}(\partial M) \xrightarrow{\EV} H^1(\partial M; \bbR) \times H^1(\partial M;\bbZ)
\end{equation*}
The closure of $\text{EV} \circ i^*( PH_{\widetilde{G}}(M))$ in $H^1(\partial M; \bbR) \times H^1(\partial M;\bbZ)$ is called the holonomy extension locus of $M$ and denoted $HL_{\widetilde{G}}(M).$
\end{defi}

\begin{defi}
\label{def:Hyp/Para/CentralPoints}
We call a point in $HL_{\widetilde{G}}(M)$ a hyperbolic/parabolic/central point if it comes from a representation $\widetilde{\rho}\in PH_{\widetilde{G}}(M)$ such that $i^*(\widetilde{\rho})$ is hyperbolic/parabolic/central. We call points in $HL_{\widetilde{G}}(M)$ but not in $\text{EV} \circ i^*( PH_{\widetilde{G}}(M))$ ideal points.
\end{defi}

To get concrete coordinates on the holonomy extension locus as well as the Dehn surgery space, let us pick a basis $(\mu, \lambda)$ for $H_1(\partial M;\bbR)$ where $\lambda$ is the homological longitude of $M$. We identify $H^1(\partial M;\bbR)$ with $\bbR^2$ using the dual basis $(\mu^*,\lambda^*)$. Let $L_r$ be the line through the origin in $\bbR^2$ of slope $-r$ where $r \in \bbQ \cup \{\infty\}$. In terms of the dual basis $(\mu^*,\lambda^*)$, the line $L_r$ consists of linear functions that vanish on the primitive element $\gamma$ representing the slope $r$ in $\pi_1(\partial M)$ with respect to the basis $(\mu,\lambda)$. The structure of the holonomy extension locus is summarized as follows:

\begin{thm}\cite[Theorem 3.1]{gao2019orderability}
The holonomy extension locus \[HL_{\tilde{G}}(M) = \bigsqcup_{i,j\in \bbZ} H_{i,j}(M)\] is a locally finite union of analytic arcs and isolated points. Each component $H_{i,j}(M)$ contains at most one parabolic point and has finitely many ideal points locally. The locus $H_{0,0}$ contains the horizontal axis $L_0$, which comes from representations to $\tilde{G}$ with abelian image. 
\end{thm}

The holonomy extension locus gives a tool to detect left-orderable Dehn surgies. We have the following lemma: 

\begin{lem}\cite[Lemma 3.8]{gao2019orderability}
\label{lem:OrderingTool}
If $L_r$ intersects the component $H_{0,0}(M)$ of $HL_{\widetilde{G}}(M)$ at non-parabolic and non-ideal points, and assume that $M(r)$ is irreducible, then $M(r)$ is left-orderable.
\end{lem}

Using the previous lemma, Xinghua Gao gives a criterion in terms of the $\PSL_2(\bbR)$-character variety to produce an interval of left-orderable Dehn surgery around the $0$-filling.

\begin{thm} \cite[Theorem 5.1]{gao2019orderability}
\label{thm:Gao-Ordering}
Suppose that $M$ is a longitudinally rigid irreducible $\bbQ$-homology solid torus and that the Alexander polynomial of $M$ has a simple positive real root $\xi \neq 1$. Then there exists a nonempty interval $(-a,0]$ or $[0,a)$ such that for every rational $r$ in the interval, the Dehn filling $M(r)$ is orderable.
\end{thm}


For completeness, we include the proof of this theorem. The key to the proof of \cref{thm:Gao-Ordering} is to produce an arc in $H_{0,0}(M)$ transverse to the horizontal axis. By construction, this arc does not contain any parabolic or ideal points. \cref{thm:Gao-Ordering} then follows from \cref{lem:OrderingTool}. To construct an arc in $H_{0,0}(M)$, we start by deforming abelian representations coming from the roots of the Alexander polynomial into irreducible representations. In particular, let $\xi$ be a simple positive real root of the Alexander polynomial and $\alpha: \pi_1(M) \to \bbR_+$, the multiplicative group of the real numbers, such that $\alpha$ factors through $H_1(M;\bbZ)_{\text{free}} \cong \bbZ$ and takes a generator of $H_1(M;\bbZ)_{\text{free}}$ to $\xi$. We let $\rho_\alpha:\pi_1(M) \to  G_\bbC$ be the associated diagonal representation given by

\begin{equation}
\label{eq:AlexanderAbelian}    
 \rho_\alpha(\gamma) = \pm \begin{pmatrix}
 \alpha(\gamma)^{1/2} & 0 \\
 0 &  \alpha(\gamma)^{-1/2}
 \end{pmatrix}
\end{equation}
where $\alpha(\gamma)^{1/2}$ is either square root. The condition on the root of the Alexander polynomial allows one to deform $\rho_\alpha =: \rho_0$ into an analytic path of representations $\rho_t :\pi_1(M) \to G$ where $t\in [-1,1]$, see \cite[Lemma 5.1]{gao2019orderability}. Furthermore this path of representations has the following properties.

\begin{lem}\cite[Lemma 5.1]{gao2019orderability}
\label{lem:GaoLemma_5_1}
The path $\rho_t:[-1,1] \to R_G(M)$ constructed above satisfies:
\begin{enumerate}
    \item The representations $\rho_t$ are irreducible over $G_\bbC$ for $t\neq 0$.
    \item The corresponding path $[\rho_t]$ of characters in $X_G(M)$ is also a non-constant analytic path.
    \item The function $\tr^2(\gamma)$ is nonconstant in $t$ for some $\gamma \in \pi_1(\partial M)$. 
\end{enumerate}
\end{lem}

\begin{proof}[Proof of Theorem \ref{thm:Gao-Ordering}] Let $\rho_t$ be the path of representations from \cref{lem:GaoLemma_5_1}. Using this path, we can produce an arc in $H_{0,0}(M)$ as follows. Since $\rho_0$ factors through $H_1(M;\bbZ)_{\text{free}} \cong \bbZ$, we can lift this representation to $\tilde{\rho}_0:\pi_1(M) \to \widetilde{G}$. As the obstruction of lifting a representation from $G$ to $\widetilde{G}$ has discrete values and is continuous on $R_G(M)$, we can lift the path $\rho_t$ to a path $\widetilde{\rho}_t$ in $R_{\widetilde{G}}(M)$. Adjusting $\widetilde{\rho}_0$ by the appropriate central element of $\widetilde{G}$, we can assume that $\trans(\widetilde{\rho}_0(\mu)) = 0$. The image $\rho_0(\lambda)$ is trivial implies that $\trans(\widetilde{\rho}_0(\lambda)) = 0$. Therefore,  $\widetilde{\rho}_0$ is mapped to a point on the horizontal axis of $H_{0,0}(M)$. Since $\xi \neq 1$, the $x$-coordinate of $\widetilde{\rho}_0$, $\ln(|\xi|)$, is nonzero.   

Let $k$ be the index of $[\mu]$ in $H_1(M;\bbZ)_{\text{free}}$ and $\xi \neq 1$ be the positive real root of the Alexander polynomial that corresponds to $\rho_0$. We have $\tr^2(\rho_0(\mu)) = \xi^k + 2 + \xi^{-k} > 4$. Therefore, $\rho_0(\mu)$ is hyperbolic and so is $\widetilde{\rho}_0(\mu)$. This implies that the image of the abelian subgroup $\langle \mu,\lambda\rangle$ under $\widetilde{\rho}_0$ contains only hyperbolic elements and that $\widetilde{\rho}_0\in PH_{\widetilde{G}}(M)$. Since being hyperbolic is an open condition, there exists $\varepsilon>0$ such that $\widetilde{\rho}_t \in PH_{\widetilde{G}}(M)$ for all $t \in[-\varepsilon,\varepsilon]$. The path $\widetilde{\rho}_t$ restricted to $[-\varepsilon,\varepsilon]$ projects to a path $A$ in $H_{0,0}(M)$. The projection $A$ must be non-constant since $[\rho_t]$ is a non-constant path on the character variety $X_G(M)$.

Finally, we use the hypothesis that $M$ is longitudinally rigid to argue that $A$ is not contained in the horizontal axis of $H_{0,0}(M)$. If $A$ is contained in $L_0$, then $\ln|a| = 0$ where $a$ is the eigenvalue of $\rho_t(\lambda)$. Since $\rho_t(\lambda)$ is either hyperbolic or trivial, we must have $\rho_t(\lambda)$ is trivial in $G$. Therefore, the representation $\rho_t$ factors through $M(0)$. We get a non-constant path $[\rho_t]$ in $X(M(0))$ of irreducible characters when $t\neq 0$. However, this contradicts the assumption that $M$ is longitudinally rigid. 
\end{proof}



\begin{remark}
The condition that $M$ is longitudinally rigid ensures that the representation $\rho_t$ obtained by deforming the abelian representation $\rho_0$ does not factor through the longitudinal filling. We can weaken this hypothesis by a local condition at the non-abelian reducible representation $\rho^+_\xi$ that corresponds to a root $\xi$ of the Alexander polynomial.
\end{remark}

Recall that, we have the following theorem of Burde and de Rham:
\begin{thm}[\cite{Burde67} and \cite{dR67}]
\label{thm:BurdeAnddeRham}
Let $\alpha:\pi_1(M) \to \bbC^*$ be a representation and define $\rho_\alpha$ as in \cref{eq:AlexanderAbelian}. Then there exists a reducible, non-abelian representation $\rho^+_\xi:\pi_1(M) \to \PSL_2(\bbC)$ such that $[\rho^+_\xi] = [\rho_\alpha]$ in $X(M)$ if and only if $\alpha$ factors through $H_1(M;\bbZ)_{\text{free}} \cong \bbZ$ sending a generator to the root $\xi$ of the Alexander polynomial of $M$.
\end{thm}

\begin{defi}
\label{def:loclongitudinallyrigid}
Suppose that $M$ be an irreducible $\bbQ$-homology solid torus. Let $\xi$ be a root of the Alexander polynomial of $M$ and $\rho^+_\xi$ be a non-abelian reducible representation associated to $\xi$. We say that $M$ is \emph{locally longitudinally rigid at $\xi$} if
\begin{equation*}
    H^1(M(0);\mathfrak{sl}_2(\bbC)_{\rho^+_\xi}) = 0.
\end{equation*}
\end{defi}



Before proving \cref{thm:Le-Ordering}, we need the following lemmas from \cite{HP05_NonAbelianReducible} in the real setting. We include the proof of these lemmas for completeness. 

\begin{lem}
\label{lem:InjectivityOnSecondCohomology}
Let $\xi$ be a simple positive real root of the Alexander polynomial that is not 1 and 
\[\phi:=\rho^+_\xi: \pi_1(M) \to \PSL_2(\bbR)\] be a non-abelian reducible representation that corresponds to $\xi$. Then the map \[H^2(\pi_1(M);\liesl_2(\bbR)_\phi) \to H^2(\pi_1(\partial M);\liesl_2(\bbR)_\phi)\] induced by the inclusion $\pi_1(\partial M) \hookrightarrow \pi_1(M)$ is injective. 
\end{lem}

\begin{proof}
We have $\phi|_{\pi_1(\partial M)}$ is nontrivial since $\tr^2(\phi(\mu)) = \xi^k + 2 + \xi^{-k} > 4$ where $k$ is the index $\langle[\mu]\rangle$ in $H_1(M;\bbZ)_{\text{free}}$. Since $\partial M$ is aspherical, we have $H^*(\partial M;\liesl_2(\bbR)_\phi) \cong H^*(\pi_1(\partial M);\liesl_2(\bbR)_\phi)$. Since $\phi|_{\pi_1(\partial M)}$ is non-trivial, we have 
\[H^0(\partial M;\liesl_2(\bbR)_\phi) \cong \liesl_2(\bbR)^{\phi(\pi_1(\partial M))}\cong \bbR.\]
By duality and Euler characteristic, we have 
\[
H^2(\partial M;\liesl_2(\bbR)_\phi) \cong \bbR 
\quad \text{and} \quad
H^1(\partial M;\liesl_2(\bbR)_\phi) \cong \bbR^2.
\]
Since $\xi$ is a simple root of the Alexander polynomial, \cite[Corollary 5.4]{HP05_NonAbelianReducible} gives that \[H^1(M;\liesl_2(\bbR)_\phi)\cong H^1(\pi_1(M);\liesl_2(\bbR)_\phi) \cong \bbR.\] By duality, we have
\[
H^2(M,\partial M;\liesl_2(\bbR)_\phi) \cong H^1(M;\liesl_2(\bbR)_\phi) \cong \bbR.
\]Therefore, the following segment of the long exact sequence of pair for $(M,\partial M)$
\[
H^1(M;\liesl_2(\bbR)_\phi) \to H^1(\partial M;\liesl_2(\bbR)_\phi) \to 
H^2(M,\partial M;\liesl_2(\bbR)_\phi)
\]
is short exact. Therefore from the long exact sequence of pair for $(M,\partial M)$ we see that the map
\[
H^2(M;\liesl_2(\bbR)_\phi) \to H^2(\partial M;\liesl_2(\bbR)_\phi) 
\]
is injective. The conclusion of the lemma follows from the following commutative diagram
\begin{center}
\begin{tikzcd}
H^2(M;\liesl_2(\bbR)_\phi) \arrow[r,hook] & H^2(\partial M;\liesl_2(\bbR)_\phi)\\
H^2(\pi_1(M);\liesl_2(\bbR)_\phi) \arrow[r]\arrow[u] & H^2(\pi_1(\partial M);\liesl_2(\bbR)_\phi)\arrow[u, "\cong"]
\end{tikzcd}    
\end{center}
and the fact that $H^2(\pi_1(M);\liesl_2(\bbR)_\phi) \to H^2(M;\liesl_2(\bbR)_\phi)$ is injective, see \cite[Lemma 3.1]{HP05_NonAbelianReducible}. 
\end{proof}

\begin{lem}
\label{lem:AllCocyclesAreIntegrable}
Let $\xi$ be a simple positive real root of the Alexander polynomial that is not 1 and 
\[\phi:=\rho^+_\xi: \pi_1(M) \to \PSL_2(\bbR)\] be a non-abelian reducible representation that corresponds to $\xi$. All cocycles in $Z^1(\pi_1(M);\liesl_2(\bbR)_\phi)$ are integrable. 
\end{lem}

\begin{proof}
As noted in the proof of \cref{lem:InjectivityOnSecondCohomology}, $\phi|_{\pi_1(\partial M)}$ is non-trivial. Since $\phi(\pi_1(\partial M)) \subset \PSL_2(\bbR)$, the image $\phi(\pi_1(\partial M))$ cannot be the Klein 4-group. By \cite[Lemma 7.4]{HP05_NonAbelianReducible}, $\phi|_{\pi_1(\partial M)}$ is a smooth point of an irreducible component of $R_G(\bbZ^2)$ with local dimension four. 

Let $i:\pi_1(\partial M) \to \pi_1(M)$ be an inclusion map and $u_1:\pi_1(M) \to \liesl_2(\bbR)$ be a cocycle. Suppose we have cochains $u_2,\dots,u_k:\pi_1(M)\to \liesl_2(\bbR)$ such that 
\[
\phi_k(\gamma) = \exp\left(\sum_{i=1}^{k}t^i u_i(\gamma)\right) \phi(\gamma)
\]
is a homomorphism modulo $t^{k+1}$. From \cref{prop:FormalDeformationAndObstructionClass}, we get an obstruction class
\[
\zeta_{k+1}^{(u_1,\dots,u_k)} \in H^2(\pi_1(M);\liesl(\bbR)_\phi),
\]
which vanishes if and only if $\phi_k$ can be extended to a homomorphism modulo $t^{k+2}$.

The restriction $\phi_k\circ i$ is a homomorphism modulo $t^{k+1}$. Since $\phi \circ i$ is a smooth point of $R_G(\bbZ^2)$, $\phi_k \circ i$ extends to a homomorphism modulo $t^{k+2}$, see \cite[Lemma 3.7]{HPS}. Therefore, the order $k+1$ obstruction vanishes on the boundary:
\[
i^*\zeta_{k+1}^{(u_1,\dots,u_k)} = \zeta_{k+1}^{(i^*u_1,\dots,i^*u_k)} = 0. 
\]
By \cref{lem:InjectivityOnSecondCohomology}, $i^*$ is injective, and so the obstruction $\zeta_{k+1}^{(u_1,\dots,u_k)}$ vanishes for $\pi_1(M)$ as well. Iterating this process starting with $u_1$, we get an infinite sequence of cochains $\{u_i\}_{i=1}^\infty$ such that $u_1$ is a cocycle and the obstruction
\[
\zeta_{k+1}^{(u_1,\dots,u_k)} =0
\]
for all $k\geq 1$. By \cref{prop:FormalDeformationAndObstructionClass}, we get a representation $\phi_\infty: \pi_1(M) \to \PSL_2(\bbR[[t]])$
\[
\phi_\infty(\gamma) = \exp\left(\sum_{i=1}^{\infty}t^iu_i(\gamma)\right)\phi(\gamma)
\]
for all cocycle $u_1$. Therefore, all cocycles of $Z^1(\pi_1(M);\liesl_2(\bbR))$ are integrable. 
\end{proof}

\begin{remark}
This strategy of proving that all cocycles are integrable was carried out over $\bbC$ in \cite[Lemma 7.5]{HP05_NonAbelianReducible}. The key tool is \cite[Lemma 3.7]{HPS} which uses the formal implicit function theorem. Since the formal implicit function theorem holds over $\bbR$, we can also carry out this strategy over $\bbR$. The same strategy to prove that certain cocyles are integrable over $\bbR$ was also carried out in the proof of \cite[Propsition 10.2]{HP05_NonAbelianReducible}.
\end{remark}

\begin{proof}[Proof of \cref{thm:Le-Ordering}]
Following the proof of \cref{thm:Gao-Ordering}, it suffices to prove that the arc $A$ constructed in the proof of \cref{thm:Gao-Ordering} is not contained in $L_0$. Arguing by contradiction, suppose this arc is contained in $L_0$. As in the proof of \cref{thm:Gao-Ordering}, this would imply that the path of representation $\rho_t$ factors through $M(0)$. Since $\rho_t$ is irreducible for all $t \neq 0$, we obtain an arc in $X(M(0))$ that contains $[\rho_\alpha] = [\rho_\xi^+]$. 

On the other hand, we claim that there exists a path $\phi_t:[-1,1] \to R_G(M)$ such that $\phi_0 = \rho^+_\xi$, the non-abelian reducible representation that corresponds to $\xi$. For convenience, we let $\phi:= \rho^+_\xi$. We have the following isomorphism of cohomology groups
\[
H^1(\pi_1(M);\liesl_2(\bbC)_\phi)  = H^1(\pi_1(M);\liesl_2(\bbR)_\phi) \otimes_\bbR \bbC. 
\]
Since $\xi$ is a simple root of the Alexander polynomial, \cite[Corollary 5.4]{HP05_NonAbelianReducible} gives that $H^1(\pi_1(M);\liesl_2(\bbR)_\phi)$ is one-dimensional. Therefore, $Z^1(\pi_1(M);\liesl_2(\bbR)_\phi)$ is four-dimensional. By \cref{lem:AllCocyclesAreIntegrable}, all cocycles in $Z^1(\pi_1(M);\liesl_2(\bbR)_\phi)$ are integrable. Therefore, $\phi$ is a smooth point of $R_G(M)$ with local dimension 4. Integrating a cocycle that generates $H^1(\pi_1(M);\liesl_2(\bbR)_\phi)$, we obtain a path $\phi_t:[-1,1] \to R_G(M)$ which has $\phi_0 = \phi$ and is transverse to the orbit of $\phi$ at $t = 0$.  

We note that since $\phi$ is a smooth point of $R_G(M)$, it is contained in a unique irreducible component of $R_G(M)$. Since the abelian representations of $\pi_1(M)$ form an irreducible component of dimension 3, $\phi$ is locally four-dimensional implies that the path $\phi_t$ cannot contain any abelian representation. Consider the path $[\phi_t]$ in $X_G(M)$ which contains the character $[\phi_0]=[\rho_0]$ coming from an abelian representation. By \cite[Proposition 10.2]{HP05_NonAbelianReducible}, $[\rho_0]$ is contained in precisely two real curves of characters. One of the curves is associated with abelian representations, and the other one with irreducible representations. Since $\phi_t$ is non-abelian representation for all $t$, the path $[\phi_t]$ is contained in the curve of irreducible characters. Therefore, for $t\neq 0$, the character $[\phi_t]$ is the character of some irreducible representation. Up to shrinking either $\rho_t$ or $\phi_t$, we may assume that $[\phi_t] = [\rho_t]$ for all $t$. Since $\phi_t$ has the same character as an irreducible representation $\rho_t$ for $t\neq 0$, the representation $\phi_t$ is conjugate to $\rho_t$ for all $t \neq 0$ by \cite[Proposition 1.5.2]{CS83}. Since $\rho_t$ factors through $M(0)$ for all $t$, we also get that $\phi_t$ factors through $M(0)$ for all $t$. We obtain a path $[\phi_t]$ in $R_G(M(0))$ going through $\phi = \rho^+_\xi$ that is transverse to the orbit of $\rho^+_\xi$. The existence of this path implies that
\[
\dim_\bbR Z^1(\Gamma(0),\liesl_2(\bbR)_{\rho_\xi^+}) \geq \dim_\bbR T^{Zar}_{\rho_\xi^+}(R(\Gamma(0))) \geq 1 + \dim_\bbR B^1(\Gamma(0),\liesl_2(\bbR)_{\rho_\xi^+}) = 4.
\]
This would imply that $\dim_\bbR H^1(\Gamma(0),\liesl_2(\bbR)_{\rho_\xi^+}) \geq 1$. Since $\liesl_2(\bbC) = \liesl_2(\bbR) \otimes_\bbR \bbC $, we have 
\[
H^1(\Gamma(0),\liesl_2(\bbC)_{\rho_\xi^+}) = H^1(\Gamma(0),\liesl_2(\bbR)_{\rho_\xi^+}) \otimes_\bbR \bbC.
\]
Therefore, the dimension of $H^1(\Gamma(0),\liesl_2(\bbC)_{\rho_\xi^+})$ is at least 1. This gives a desired contradiction to the condition that $M$ is locally longitudinally rigid at $\xi$. 
\end{proof}

\section{The [1,1,2,2,2j] two-bridge knots}
\label{sec:TwoBridge}
In this section, we apply \cref{thm:Le-Ordering} to study left-orderability on the family of two-bridge knots $K_j$ associated to the continued fraction $[1,1,2,2,2j]$ for $j\geq 1$ and prove \cref{thm:LOTwoBridge}. We first make some remarks about this family of two-bridge knot complements. 

These knot complements are obtained by doing $1/j$ Dehn filling on the unknot component of the link $L^2_{25}$, see \cref{fig:L2_25}. The first two members of the family are the knots $8_{12}$ and $10_{13}$ in Rolfsen's table. As we will see in \cref{lem:AlexanderPoly}, the Alexander polynomial of $K_j$ has all simple positive real roots, that are not 1, and is not monic for $j \geq 2$. In particular, the complement of $K_j$ is not lean for $j\geq 2$. Furthermore, the trace field of $K_j$ for $1 \leq j \leq 30 $ has no real places, and it is most likely that the trace fields of all knots in this family share this property. Therefore, \cref{thm:LOTwoBridge} is not a direct consequence of \cref{thm:CullerDunfieldGao} nor \cref{thm:CullerDunfieldTraceFieldOrdering}. The family of two-bridge knots $[1,1,2,2,2j]$ is a new family of knots with an interval left-orderable Dehn surgeries which cannot be obtained from prior techniques. 

\begin{figure}
    \centering
    \includegraphics[scale = 0.5]{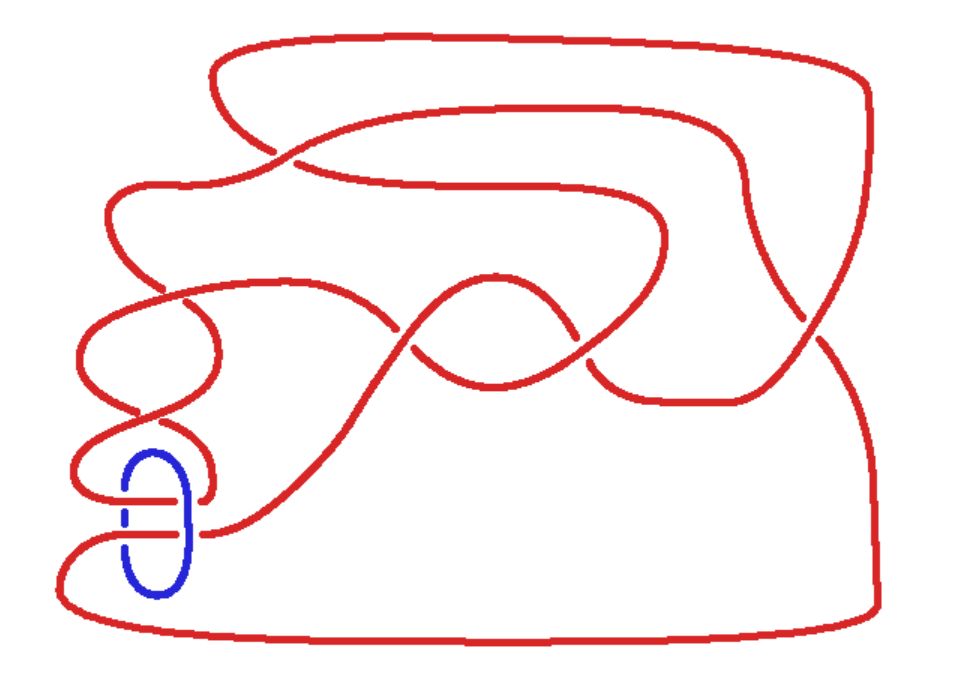}
    \caption{The link complement $L^2_{25}$}
    \label{fig:L2_25}
\end{figure}

\subsection{Group presentation}
We will denote by $\Gamma$ the fundamental group of the complement of the knot $K_j$. The knot corresponding to the continued fraction $[1,1,2,2,2j]$ has the associated fraction
\[
[1,1,2,2,2j] = 1\frac{1}{1+\frac{1}{2+\frac{1}{2+\frac{1}{2j}}}} = \frac{24j + 5}{14j + 3}.
\]
By \cite[Proposition 1]{R72}, the knot group $\Gamma$ has the presentation $ \Gamma = \langle x,y \mid xw = wy\rangle$. The word $w$ is given by
\begin{equation}
    \label{eq:wxy}
    w = y^{e_1}x^{e_2}\dots y^{e_{24j+3}}x^{e_{24j+4}}
\end{equation}
where $e_i = (-1)^{\lfloor i(14j+3)/(24j+5)\rfloor}$. Also by \cite[Proposition 1]{R72}, the homological longitude of $K_j$ that commutes with $x$ is given by $\ell = wv$ where
\begin{equation}
    \label{eq:vxy}
    v = x^{e_{24j+4}} y^{e_{24j+3}} \dots x^{e_2}y^{e_1}.
\end{equation}
We first give an explicit description of $w$ in terms of $x$ and $y$ by giving a formula for the right-hand sides of \cref{eq:wxy} and \cref{eq:vxy}. We have the following lemma. 

\begin{lem}
In the terms of the generators $x,y$ of $\Gamma$, the word $w$ has the form
\begin{equation}
\label{eq:wj}
    w = (yx^{-1}y^{-1}x)u^j \quad \text{and} \quad v = s^j(xy^{-1}x^{-1}y)
\end{equation}
where 
\begin{equation*}
\label{eq:u}
    u = (yx^{-1}yx)(y^{-1}x^{-1}yx^{-1})(y^{-1}xyx^{-1})(y^{-1}xy^{-1}x^{-1})(yxy^{-1}x)(yx^{-1}y^{-1}x)
\end{equation*}
and $s$ is $u$ spelled backwards.
\end{lem}

\begin{proof}
Since $v$ is $w$ spelled backwards, it suffices to prove the lemma for $w$. Let us consider  
\[k_{i,j} = \frac{i(14j+3)}{24j + 5}\]
for $1 \leq i \leq 24j + 4$. We first claim that 
\begin{equation}
\label{eq:kijPieceWiseConstant}
\lfloor k_{i,j} \rfloor = \lfloor k_{i,m} \rfloor = \left\lfloor \frac{7i}{12}\right\rfloor    
\end{equation}
for all $j\geq m$ and $\varepsilon_m \leq  i \leq 24m + 4$ where $\varepsilon_m = \max\{1,24(m-1) + 5\} $. Fixing $i$, we can view $k_{i,j}$ as a continuous function in the variable $j$. Since $i\geq 1$, the derivative of $k_{i,j}$ with respect to $j$ is 
\begin{equation*}
    \frac{d k_{i,j}}{dj} = -\frac{2i}{(24j+5)^2} < 0.
\end{equation*}
The function $k_{i,j}$ is strictly decreasing with and has a horizontal asymptote at $7i/12$ as $j\to +\infty$. Therefore, we have the following chain of inequalities 
\[\frac{7i}{12} < k_{i,j} < k_{i,m} = \frac{(14 m + 3) i }{24 m + 5} \]
for all $j\geq m$ and $\varepsilon_m \leq i \leq 24 m + 4 $. We have 
\[
0< k_{i,m} - \left\lfloor\frac{7i}{12}\right\rfloor \leq k_{i,m} - \frac{7i}{12} +\frac{11}{12} =\frac{264m + 55 + i}{288m + 60} < 1
\]
for all $ \varepsilon_m\leq i\leq 24m+4$. It follows that $k_{i,j}$ is contained in the interval $(\left\lfloor\frac{7i}{12}\right\rfloor,\left\lfloor\frac{7i}{12}\right\rfloor+1]$ for all $j\geq m$ and $\varepsilon_m \leq i \leq 24 m + 4 $. To verify \cref{eq:kijPieceWiseConstant}, it remains to show that $k_{i,m}$ is not an integer for all $\varepsilon_m \leq i \leq 24m + 4$. Since $14m+3$ and $24m+5$ are relatively prime, $k_{i,m}$ is an integer if and only if $24m+5$ divides $i$. But this is not possible since $1\leq i < 24m+5$. 

By a direct computation, we can verify \cref{eq:wj} when $j=1$. From \cref{eq:kijPieceWiseConstant}, we see that the right-hand side of \cref{eq:wxy} has prefix $w_1 = yx^{-1}y^{-1}xu$ for all $j \geq 1$. We write $w = w_1 w_j' = (yx^{-1}y^{-1}x)u w_j'$. It remains to show that $w_j' = u^{j-1}$. Using \cref{eq:kijPieceWiseConstant}, we have
\[
     \lfloor k_{i+24n,j} \rfloor = \left\lfloor \frac{7i}{12} + 14n \right\rfloor =\left\lfloor \frac{7i}{12}  \right\rfloor + 14n =   \lfloor k_{i,j} \rfloor + 14n
\]
for all $5 \leq i \leq 28$ and $5 \leq i + 24n \leq 24j + 4$. We have
\begin{equation}
\label{eq:kijParity}
\lfloor k_{i,j} \rfloor \equiv \lfloor k_{i + 24n,j} \rfloor \mod 2    
\end{equation}
for all $5 \leq i \leq 28$ and $5 \leq i + 24n \leq 24j + 4$. \cref{eq:kijParity} implies that the parity of $\lfloor k_{i,j} \rfloor$ repeats with period 24 when $i\geq 5$. Since the word for $w$ in $x$ and $y$ only depends on this parity, the word $w$ is given by \cref{eq:wj} as claimed. This completes the proof of the lemma.





\end{proof}

\subsection{The Alexander polynomial of $K_j$}
Now we will compute the Alexander polynomial of $K_j$ using non-abelian reducible representations. Let $\rho:\Gamma \to \SL_2(\bbC)$ be a non-abelian reducible representation of $\Gamma$. Since $\Gamma$ is generated by two conjugate meridians $x$ and $y$, the representation $\rho$ can be conjugated to have the form 

\begin{equation}
    \label{eq:NonAbelianReducible}
    x \mapsto \begin{pmatrix}
    t & 0 \\ 0 & t^{-1}
    \end{pmatrix}, 
    \quad \text{and} \quad
    y \mapsto \begin{pmatrix}
    t & 1 \\ 0 & t^{-1}
    \end{pmatrix}
\end{equation}
where $t\neq \pm 1$. By \cref{thm:BurdeAnddeRham}, for a knot group $\Gamma$ the assignment in \cref{eq:NonAbelianReducible} defines a representation of $\Gamma$ if and only if $t^2$ is a root of the Alexander polynomial $\Delta(\tau) \in \bbZ[\tau^{\pm 1}]$. Consequently, we can use this fact to compute the Alexander polynomial of the knot $K_j$ as follows. 

Let $F_2$ be the free group on two letters $X$ and $Y$. Consider the representation $P:F_2 \to \SL_2(\bbZ[\tau^{\pm 1}])$
\[
 X \mapsto 
    \begin{pmatrix} 
    \tau & 0 \\ 0 & \tau^{-1} 
    \end{pmatrix}, 
    \quad \text{and} \quad 
Y \mapsto 
    \begin{pmatrix}
    \tau & 1 \\ 0 & \tau^{-1}
    \end{pmatrix}.
\]
Let $W$ be the word in $X$ and $Y$ given by \cref{eq:wj}. A direct calculation shows that
\[
P(W) = \begin{pmatrix}
1 & -j \tau^3 + (5j+ 1) \tau - (5j+1)\tau^{-1} +j \tau^{-3}\\ 0 & 1 
\end{pmatrix}
\]
The representation $P$ factors through the natural projection $F_2 \to \Gamma$ if and only if $P(XW)  = P(WY)$. Or equivalently, we have
\[
j\tau^4 - (6j+1)\tau^2 + (10j+3) - (6j+1) \tau^{-2} + j \tau^{-4} = 0.
\]
The expression above is the Alexander polynomial of $K_j$ evaluated at $\tau^2$. As a convention, we will normalize the Alexander polynomial so that the lowest term of $\Delta(\tau)$ is a non-zero constant term. We have the following lemma. 

\begin{lem}
\label{lem:AlexanderPoly}
The Alexander polynomial of $K_j$ has the form 
\begin{equation}
\label{eq:AlexanderPoly}
    \Delta(\tau) = j \tau^4 -(6j+1) \tau^3 + (10j+3)\tau^2 -(6j+1)\tau +j.
\end{equation}
Furthermore, $\Delta(\tau)$ has exactly 4 simple real roots.
\end{lem}

\begin{proof}
The discussion prior to the lemma implies that
\[
\Delta(\tau^2) = j\tau^8 - (6j+1)\tau^6 + (10j+3)\tau^4 - (6j+1) \tau^{2} + j.  
\]
This gives us \cref{eq:AlexanderPoly} as claimed. For the claim about the roots of $\Delta$, we consider $\delta(\tau) = \Delta(\tau)/j$. We note that
\begin{equation*}
    \delta_j(0) = 1, \ 
    \delta_j(1/2) = \frac{-3j+2}{16j}, \ 
    \delta_j(1) = \frac{1}{j}, \
    \delta_j(2) = \frac{-3j+2}{j}, \ 
    \delta_j(5) = \frac{96j -55}{j}.
\end{equation*} 
For all $j \geq 1$, we see that $\delta_j$ changes signs 4 times in the interval $[0,5]$. By continuity, $\delta_j(\tau)$ has 4 distinct real roots in the interval $[0,5]$. Therefore, $\Delta(\tau)$ has at least 4 positive real roots. Since $\Delta$ has degree 4, $\Delta$ has precisely 4 simple positive real roots for all $j \geq 1$. 
\end{proof}

\subsection{The group cohomology $H^1(\Gamma(0);\liesl_2(\bbC)_\rho)$}
In this section, we will prove that the knots $K_j$ are locally longitudinal rigid by directly computing the group cohomology with coefficients in $\liesl_2(\bbC)$. We first identify $\liesl_2(\bbC)$ with $\bbC^3$ by choosing the following basis 
\begin{equation*}
    v_+ = \begin{pmatrix}
    0 & 1 \\ 0 & 0 
    \end{pmatrix}, \quad  v_0 = \begin{pmatrix}
    1 & 0 \\ 0 & -1 
    \end{pmatrix}, \quad \text{and} \quad 
     v_- = \begin{pmatrix}
    0 & 0 \\ 1 & 0 
    \end{pmatrix}.
\end{equation*}
With respect to this basis, the adjoint representation $\Ad: \SL_2(\bbC) \to \SL_3(\bbC)$ becomes
\begin{equation*}
\begin{pmatrix}
a & b \\ c & d 
\end{pmatrix}   
\mapsto
\begin{pmatrix}
a^2 & -2ab & -b^2\\
-ac & ad + bc & bd \\
-c^2 & 2cd & d^2\\
\end{pmatrix}. 
\end{equation*}

By \cref{lem:AlexanderPoly}, we can choose $t \in \bbR$ such that $t^2$ is a simple root of the Alexander polynomial $\Delta(\tau)$. Since the longitude $\ell$ belongs to the second commutator subgroup of $\Gamma$, any non-abelian reducible representation on $\Gamma$ factors through $\Gamma(0)$. We get a non-abelian reducible representation $\rho:\Gamma(0)\to \SL_2(\bbC)$ given by \cref{eq:NonAbelianReducible}. For convenience, we will write
\begin{equation*}
    \rho(w) = \begin{pmatrix} 1 & f \\ 0 & 1
    \end{pmatrix}
\end{equation*}
where $f = -j t^3 +(5j+1)t - (5j+1)t^{-1} + jt^{-3}$. The action of $\Gamma(0)$ on $\liesl_2(\bbC)$ is given by 
\begin{equation*}
x\mapsto 
\begin{pmatrix} 
t^2 & 0 & 0 \\ 0 & 1 & 0 \\ 0 & 0 & t^{-2} 
\end{pmatrix}
\quad \text{and} \quad
y\mapsto 
\begin{pmatrix} 
t^2 & -2t & -1 \\ 0 & 1 & t^{-1} \\ 0 & 0 & t^{-2} 
\end{pmatrix}.
\end{equation*}
Using \cref{eq:Coboundary}, we see that the space of coboundaries can be parametrized by $d: \Gamma(0) \to  \liesl_2(\bbC)_{\rho}$ such that
\begin{equation}
\label{eq:ParametrizedCoboundary}
d(x) = 
\begin{pmatrix}(t^2-1) a \\ 0 \\ (t^{-2}-1) c \end{pmatrix} 
\quad \text{and} \quad 
d(y) = 
\begin{pmatrix}
(t^2-1) a   -2tb - c \\ t^{-1}c \\ (t^{-2}-1)c
\end{pmatrix}.
\end{equation}

\begin{prop}
\label{prop:ParametrizingH1}
Any cohomology class in $H^1(\Gamma(0);\liesl_2(\bbC)_{\rho})$ can be represented by a 1-cocycle $z\in Z^1(\Gamma(0);\liesl_2(\bbC)_\rho)$ such that
\begin{equation}
\label{eq:NormalizedCocycle}
z(x) = \begin{pmatrix}
    0 \\ \alpha \\ \beta
\end{pmatrix}
\quad \text{and} \quad
z(y) = \begin{pmatrix}
0 \\ \alpha \\ 0 
\end{pmatrix}.
\end{equation}
\end{prop}

\begin{proof}
Let $z\in Z^1(\Gamma(0);\liesl_2(\bbC)_\rho)$ be a 1-cocycle. Since $t^2 \neq 1$, by an appropriate choice of $a,b,c \in \bbC$ for a coboundary $d$ in \cref{eq:ParametrizedCoboundary}, we can assume that
\begin{equation*}
    z(x) = \begin{pmatrix}
    0 \\ \alpha \\ \beta
\end{pmatrix}
\quad \text{and} \quad
z(y) = \begin{pmatrix}
0 \\ \delta \\ 0 
\end{pmatrix}.
\end{equation*}
The relation $z(xw) = z(wy)$ implies that 

\[z(x) + (x - 1)\cdot z(w) - w \cdot z(y) = 0.\]
Or equivalently, we have
\begin{equation*}
    \begin{pmatrix}
    0 \\ \alpha \\ \beta
    \end{pmatrix} + 
    \begin{pmatrix}
    t^2 - 1 & 0 & 0 \\ 0 & 0 & 0 \\ 0 & 0 & t^{-2}- 1 
    \end{pmatrix}
    \begin{pmatrix}
    \omega_1 \\ \omega_2 \\ \omega_3 
    \end{pmatrix} - 
    \begin{pmatrix}
    1 & -2 f & -f^2 \\ 0 & 1 &  f \\ 0 & 0 & 1
    \end{pmatrix}
    \begin{pmatrix}
    0 \\ \delta \\ 0  
    \end{pmatrix} = 
    \begin{pmatrix}
    0 \\ 0 \\ 0
    \end{pmatrix}.
\end{equation*}
The second coordinate of the previous equation implies that $\delta = \alpha$. 
\end{proof}


We will need the following lemma
\begin{lem}
\label{lem:z(w)_z(v)}
Let $z\in Z^1(\Gamma(0);\liesl_2(\bbC)_\rho)$ be given by \cref{eq:NormalizedCocycle}. Suppose that 
\begin{align*}
z(w) = \omega_1 v_+ +  \omega_2 v_0 +  \omega_3v_- \quad \text{and}\quad  z(v) = \nu_1 v_+ +  \nu_2 v_0 +  \nu_3v_-.
\end{align*} Then 
\begin{align*}
    \omega_1 = &\alpha(-4jt^3 +(10j+2)t-2jt^{-3}) + \beta h \\
    \omega_2 =  &\left(\frac{1}{2}j(j+1)t^7 -5j(j+1)t^5 +\frac{1}{2}(35j^2+31j + 2)t^3 -\frac{1}{2}(52j^2+28j+4)t \right. \\
    &\left.+ \frac{1}{2}j(35j-3)t^{-1}-\frac{1}{2}j(10j-6)t^{-3} + \frac{1}{2}j(j-1)t^{-5}\right)\beta\\
    \nu_2 =  &\left(\frac{1}{2}j(j+1)t^7 -5j(j+1)t^5 +\frac{1}{2}(35j^2+27j + 2)t^3 -\frac{1}{2}(52j^2+12j)t \right. \\
    &\left.+ \frac{1}{2}j(35j-7)t^{-1}-\frac{1}{2}j(10j-6)t^{-3} + \frac{1}{2}j(j-1)t^{-5}\right)\beta\\
    \omega_3 &= -\nu_3 = tf \beta 
\end{align*} 
where $h \in \bbC$.  
\end{lem}

\begin{proof}
The proof of this lemma is a direct calculation. By a repeat application of the cocycle relation in \cref{eq:FoxDerivatives}, we have
\begin{align*}
    z(w) &= z(yx^{-1}y^{-1}x) + (yx^{-1}y^{-1}x)\cdot\sum_{i=0}^{j-1} u^i \cdot z(u),\\
    z(v) &=  s^j \cdot z(xy^{-1}x^{-1}y) + \sum_{i=0}^{j-1} s^i \cdot z(s).
\end{align*}
We also have 
\begin{align*}
    \sum_{i=0}^{j-1} (\Ad\circ \rho) (u^i) &= \begin{pmatrix}
    j & (j^2-j)(t^3-5t+5t^{-1} - t^{-3}) & -\frac{1}{6}(2j^3-3j^2 +j)(t^3-5t+5t^{-1} - t^{-3})^2 \\
    0 & j & -\frac{1}{2}(j^2-j)(t^3-5t+5t^{-1} - t^{-3}) \\
    0 & 0 & j
    \end{pmatrix}, \\
    \sum_{i=0}^{j-1} (\Ad\circ \rho) (s^i) &= \begin{pmatrix}
    j & -(j^2-j)(t^3-5t+5t^{-1} - t^{-3}) & -\frac{1}{6}(2j^3-3j^2 +j)(t^3-5t+5t^{-1} - t^{-3})^2 \\
    0 & j & \frac{1}{2}(j^2-j)(t^3-5t+5t^{-1} - t^{-3}) \\
    0 & 0 & j
    \end{pmatrix},
\end{align*}
and
\begin{equation*}
    z(u) = 
    \begin{pmatrix}
    (-4t^3 + 10t -2t^{-3})\alpha + h'\beta\\ (t^7 - 9t^5 + 27t^3 - 30t +10t^{-1} -t^{-3})\beta \\ (-t^4 +5t^2-5 +t^{-2})\beta 
    \end{pmatrix} \quad \text{and} \quad
    z(s) = 
    \begin{pmatrix}
    (4t^3 - 10t +2t^{-3})\alpha + h''\beta\\ (t^7 - 9t^5 + 25t^3 - 22t +8t^{-1} -t^{-3})\beta \\ (t^4 - 5t^2 +5 -t^{-2})\beta 
    \end{pmatrix}
\end{equation*}
for some $h',h'' \in \bbC$. The lemma will follow once we note that
\[
z(yx^{-1}y^{-1}x)  = \begin{pmatrix}
2t\alpha - (t^4-3t^2+1)\beta \\ (t^3-2t)\beta \\ (t^2-1)\beta
\end{pmatrix} \quad \text{and} \quad 
z(xy^{-1}x^{-1}y) = \begin{pmatrix}
(-6t + 2t^{-1})\alpha + t^4\beta \\
t^3\beta \\
(-t^2 + 1)\beta
\end{pmatrix}.
\]
\end{proof}

Now we are ready to show that $H^1(\Gamma(0);\liesl_2(\bbC)_\rho) = 0$.

\begin{proof}[Proof of \cref{thm:LOTwoBridge}]
Now we will show that $K_j$ is locally longitudinally rigid for all $j\geq 1$ at any root of the Alexander polynomial. Let $[z] \in H^1(\Gamma(0);\liesl_2(\bbC)_\rho)$. By \cref{prop:ParametrizingH1}, we can assume that $z$ satisfies \cref{eq:NormalizedCocycle}. Since $z(\ell) = 0$, we must have
\[
z(w) + w\cdot z(v) = 0. 
\]
Let us write $z(w) =\omega_1 v_+ +\omega_2v_0 + \omega_3v_-$ and $z(v) =\nu_1 v_+ +\nu_2v_0 + \nu_3v_-$. The second coordinate of this equation implies that $\omega_2 + \nu_2 + f \nu_3 = 0$. Using \cref{lem:z(w)_z(v)}, we have
\[
\frac{(t^4 - 1)(t^4 + j (t^4 -4t^2 + 1)^2)}{t^5}\beta = 0. 
\]
Suppose that $\beta \neq 0$. Since $t \in \bbR - \{\pm 1\}$, we have 
\begin{align*}
    t^4 + j (t^4 -4t^2 + 1)^2 = 0.
\end{align*}
Since $t\in \bbR$ and $j \geq 1$, the above equation holds if and only if 
\[
t = 0 \quad \text{and} \quad t^4 -4t^2 + 1 = 0.
\]
This is the desired contradiction. Therefore, we must have $\beta = 0$. 

From the first coordinate of the relation $z(xw) = z(wy)$, we have $(t^2-1) \omega_1 + 2f \alpha = 0$. Using $\beta = 0$ and \cref{lem:z(w)_z(v)}, this equation is equivalent to
\begin{equation*}
    (t^4- 1)(2jt^4-(6j+1)t^2+2j))\alpha = 0. 
\end{equation*}
Similarly, if $\alpha \neq 0$, we must have $(2jt^4-(6j+1)t^2+2j)) = 0 $. It follows that $t^2$ is a root of both
\[\Delta(\tau) \quad \text{and} \quad h(\tau) :=(2j\tau^2-(6j+1)\tau+2j).\]

Note that the roots of $h(\tau)$ are reciprocal of each other. Since $\Delta(\tau)$ is a reciprocal polynomial, the roots of $\Delta(\tau)$ come in reciprocal pairs. Therefore, $h(\tau)$ divides $\Delta(\tau)$. By Gauss's lemma, we can write $\Delta(\tau) = h(\tau)k(\tau)$ for $k(\tau)\in \bbZ[\tau]$. This implies that $\Delta(0) = h(0)k(0)$ or $j = 2j k(0)$. This contradicts the fact that $j \geq 1$ and $k(0) \in \bbZ$. Therefore, $\alpha = 0$ and $z$ can only be the zero cocycle. Consequently, $H^1(\Gamma(0);\liesl_2(\bbC)_\rho) = 0$ where $\rho$ is any non-abelian reducible representation of $\Gamma(0)$. In other words, the knot $K_j$ is locally longitudinally rigid at any root of the Alexander polynomial. By \cref{thm:Le-Ordering}, there exists an interval of left-orderable Dehn surgeries near 0.     
\end{proof}

\printbibliography

\end{document}